\documentclass[12pt,reqno]{amsart}
\usepackage{amsmath,amssymb,geometry,color,tikz-cd,extarrows}
\usetikzlibrary{intersections,patterns}

\usepackage[backref,pagebackref,pdftex,hyperindex]{hyperref}
\usepackage[alphabetic]{amsrefs}
\usepackage{amssymb}
\usepackage{bbm}
\usepackage{enumitem}
\usepackage{upgreek}
\usepackage{bm}

\newtheorem{proposition}{Proposition}[section]
\newtheorem{theorem}[proposition]{Theorem}
\newtheorem{lemma}[proposition]{Lemma}

\newtheorem{conjecture}[proposition]{Conjecture}

\theoremstyle{definition}
\newtheorem{remark}[proposition]{Remark}
\newtheorem{definition}[proposition]{Definition}

\newtheorem{proposition-definition}[proposition]{Proposition-Definition}

\title{Finiteness of K-moduli compactifications under the boundedness of K-semistable degenerations}

\author{Chuyu Zhou}
\address{School of Mathematical Sciences, Xiamen University, Siming South Road 422, Xiamen, Fujian 361005, China.}
\email{chuyuzhou@xmu.edu.cn, chuyuzhou1@gmail.com}

\date{} 

\thanks{2010 
	    \emph{Mathematics Subject Classification}: 14J45.
	    \newline
	    \indent 
		\emph{Keywords}: Log Fano pair, K-stability, K-moduli, Wall crossing.
        \newline
		\indent
		\emph{Competing interests}:  The author has no conflict of interest to declare.
		}

\newcommand{\Fut}{{\rm{Fut}}}

\newcommand{\ord}{{\rm {ord}}}

\newcommand{\vol}{{\rm {vol}}}

\newcommand{\red}{{\rm {red}}}

\newcommand{\wt}{{\rm {wt}}}

\newcommand{\Hilb}{{\rm {Hilb}}}
\newcommand{\Chow}{{\rm {Chow}}}
\newcommand{\PGL}{{\rm {PGL}}}

\newcommand{\CM}{{\rm {CM}}}

\newcommand{\Kss}{{\rm {Kss}}}

\newcommand{\WLF}{{\rm {WLF}}}

\newcommand{\LF}{{\rm {LF}}}

\newcommand{\Kps}{{\rm {Kps}}}

\newcommand{\qKss}{{\rm {qKss}}}

\newcommand{\bA}{\mathbb{A}}

\newcommand{\bC}{\mathbb{C}}

\newcommand{\bG}{\mathbb{G}}

\newcommand{\bN}{\mathbb{N}}

\newcommand{\bP}{\mathbb{P}}
\newcommand{\bQ}{\mathbb{Q}}
\newcommand{\bR}{\mathbb{R}}

\newcommand{\bZ}{\mathbb{Z}}

\newcommand{\mA}{\mathcal{A}}
\newcommand{\mB}{\mathcal{B}}

\newcommand{\mD}{\mathcal{D}}

\newcommand{\mG}{\mathcal{G}}

\newcommand{\mL}{\mathcal{L}}
\newcommand{\mM}{\mathcal{M}}

\newcommand{\mO}{\mathcal{O}}
\newcommand{\mP}{\mathcal{P}}

\newcommand{\mX}{\mathcal{X}}
\newcommand{\mY}{\mathcal{Y}}

\newcommand{\kY}{\mathfrak{Y}}
\newcommand{\kZ}{\mathfrak{Z}}

\newcommand{\kL}{\mathfrak{L}}

\begin{document}

\begin{abstract}
In this note, we prove the finite semi-algebraic chamber decomposition theorem for K-semi(poly)stability under the assumption of the log boundedness of K-semistable degenerations. This boundedness assumption is inspired by the expected finiteness of K-moduli compactifications.

\end{abstract}

\maketitle
\markboth{Finiteness of K-moduli compactifications}{}

\setcounter{tocdepth}{1}



\section{Introduction}\label{sec: intro}

We work over the field of complex numbers  $\bC$ throughout the article.

Let $\pi: (\mX, \mD)\to T$ be a log smooth morphism over a smooth base $T$, where $\mD$ is a smooth divisor on $\mX$. Suppose
$(\mX_t, x\mD_t)$ is a K-(poly)stable log Fano pair for all $t\in T$ and $x\in [0,1)$. By \cite{LXZ22} complemented by \cite{LZ24a}, 
for any given $c\in [0,1)$, there exists an Artin stack of finite type $\mM_c$ generically parametrizing the fibers of $\pi_c: (\mX, c\mD)\to T$, and $\mM_c$ descends to a proper good moduli space $M_c$, which can be viewed as a K-moduli compactification of $T$. Write $\phi_c: \mM_c\to M_c$ to be the descending map. We pose the following question: 

\

\textit{Question: Is the set of K-moduli compactifications $K(T):=\{M_c\ |\  c\in [0,1)\}$ finite?}

\

This question is naturally related to the wall crossing property of K-moduli. In the case where $\mD$ is proportional  to $-K_{\mX}$ over $T$, there is a finite rational interval chamber decomposition 
$0=c_0<c_1<c_2<\cdot\cdot\cdot<c_k<c_{k+1}=1 $ such that $\phi_c: \mM_c\to M_c$ does not depend on $c\in (c_i, c_{i+1})$ for any given index $0\leq i\leq k$ (e.g. \cite{ADL19, Zhou23, Zhou26}). This directly implies that the set $K(T)$ is finite. However, the question is much more subtle if $\mD$ is not proportional to $-K_{\mX}$ over $T$. In the non-proportional case, if $\vol(-K_{\mX_t}-\mD_t)>0$, we still have a finite interval chamber decomposition of $[0,1)$ such that $\phi_c: \mM_c\to M_c$ does not change when restricted to each open interval; if $\vol(-K_{\mX_t}-\mD_t)=0$ while we only focus on $c\in (0,1-\epsilon]$ for some $0<\epsilon\ll 1$, we also have a finite interval chamber decomposition of $[0, 1-\epsilon)$ such that $\phi_c: \mM_c\to M_c$ does not change when restricted to each open interval (e.g. \cite{LZ24b}). The difference is that the finite interval chamber decompositions in non-proportional case may not be rational (e.g. \cite{DJKHQ24}). Above all,  the difficulty arises when $\vol(-K_{\mX_t}-\mD_t)=0$ and we consider all $c\in [0,1)$. Regarding this, the non-proportional wall crossing theory developed in \cite{LZ24b} is not complete since we do not touch the case when the coefficient is close to the Calabi-Yau boundary.
Towards this difficulty, the expected solution should be as follows: \textit{there exists a finite interval chamber decomposition of $[0,1)$ such that $\phi_c: \mM_c\to M_c$ does not change when restricted to each open interval.} This solution will totally address the finiteness question we pose at first.

Suppose the set $K(T)$ is finite, then the set $\cup_{c\in [0,1)}\mM_c(\bC)$ is log bounded. This actually indicates the key ingredient in the program of developing a complete non-proportional wall crossing theory, i.e. the set of all K-semistable degenerations (as the coefficient varies) should be log bounded. We divide the program into two steps:

\

\textit{Step 1: Show that the set of K-semistable degenerations, i.e. $\cup_{c\in [0,1)}\mM_c(\bC)$, is log bounded;}

\

\textit{Step 2: Develop a complete non-proportional wall crossing theory for K-moduli under the log boundedness in step 1.}

\

In this paper, we work in a more general setting and complete step 2. 
Throughout, we fix two positive integers $d, k$ and a rational polytope $P\subset [0,1]^k$. We also fix a log bounded set $\mG$ of couples (see Section \ref{sec: 2} for the definition of a \textit{couple}) depending on $(d, k, P)$ such that the couple $(X, \sum_{j=1}^kD_j)$ in $\mG$ satisfies the following conditions:
\begin{enumerate}
\item $X$ is a normal projective variety of dimension $d$;
\item $D_j, 1\leq j\leq k,$ are effective Weil divisors on $X$;
\item there exists a vector $(c_1,...,c_k)\in P$ such that $(X, \sum_{j=1}^kc_jD_j)$ is a \textit{K-semistable log Fano pair}.
\end{enumerate}
We emphasize that \textit{the set $\mG$ is given log boundedness} (one could for example take it as the set of fibers of $\pi$ given at the beginning). To study the K-moduli compactification, we have to extend $\mG$ to a larger set to make it complete. We define the \textit{completion} $\widehat{\mG}$ as follows.

\begin{definition}\label{def: K-complete}
We say a couple $(Y, \sum_{j=1}^kB_j)$ is contained in $\widehat{\mG}$ if and only if there exists a family of couples $(\mY, \sum_{j=1}^k\mB_j)\to C\ni 0$ (see Section \ref{sec: 2} for \textit{a family of couples}) over a smooth pointed curve such that the following conditions are satisfied:
\begin{enumerate}
\item $(\mY_t, \sum_{j=1}^k \mB_{j,t})\in \mG$ for $t\ne 0$;
\item $(\mY_0, \sum_{j=1}^k \mB_{j,0})\cong (Y, \sum_{j=1}^kB_j)$;
\item $(\mY, \sum_{j=1}^k c_j\mB_j)\to C$ is an $\bR$-Gorenstein family of K-semistable log Fano pairs for some $(c_1,...,c_k)\in P$, i.e., $-K_{\mY}-\sum_{j=1}^kc_j\mB_j$ is $\bR$-Cartier and every fiber $(\mY_t, \sum_{j=1}^kc_j\mB_{j,t})$ is a K-semistable log Fano pair.
\end{enumerate}
\end{definition}

A priori it, we do not know whether the set $\widehat{\mG}$ is log bounded, though the set $\mG$ is  log bounded. However, we expect this is true.

\begin{conjecture}{\rm (\cite{LZ24b})}\label{conj: bdd}
The set $\widehat{\mG}$ is log bounded.
\end{conjecture}

It will be enough to concern the set $\widehat{\mG}$, as it already contains all K-semistable log Fano pairs which appear in the boundary of any K-moduli compactification of $\mG$. However, for the convenience of the theoretical development, we need to extend $\widehat{\mG}$ to a even larger auxiliary set $\widetilde{\mG}$, which is defined below by two steps.

\begin{definition}\label{def: dsm}
We gradually extend $\widehat{\mG}$ to $\widetilde{\mG}$ by two steps.

\textit{Step 1}. We first define an intermediate set $\widehat{\mG}'$ (between $\widehat{\mG}$ and $\widetilde{\mG}$).
We say a couple $(Y, \sum_{j=1}^kB_j)$ is contained in $\widehat{\mG}'$ if
there is a \textit{small birational map} $(X, \sum_{j=1}^kD_j)\dashrightarrow (Y, \sum_{j=1}^kB_j)$, where $(X, \sum_{j=1}^kD_j)\in \widehat{\mG}$ and $\sum_{j=1}^kB_j$ is the birational transform of $\sum_{j=1}^kD_j$. In other words, $\widehat{\mG}  '$ is obtained by adding to $\widehat{\mG}$ all the small modifications.

\textit{Step 2}. We then define $\widetilde{\mG}$ by a similar way as in Definition \ref{def: K-complete} by adding to $\widehat{\mG}'$ all the K-semistable degenerations. More precisely, we say a couple $(Y, \sum_{j=1}^kB_j)$ is contained in $\widetilde{\mG}$ if and only if there exists a family of couples $(\mY, \sum_{j=1}^k\mB_j)\to C\ni 0$ over a smooth pointed curve such that the following conditions are satisfied:
\begin{enumerate}
\item $(\mY_t, \sum_{j=1}^k \mB_{j,t})\in \widehat{\mG}'$ for $t\ne 0$;
\item $(\mY_0, \sum_{j=1}^k \mB_{j,0})\cong (Y, \sum_{j=1}^kB_j)$;
\item $(\mY, \sum_{j=1}^k c_j\mB_j)\to C$ is an $\bR$-Gorenstein family of K-semistable weak log Fano pairs for some $(c_1,...,c_k)\in P$, i.e., $-K_{\mY}-\sum_{j=1}^kc_j\mB_j$ is $\bR$-Cartier and every fiber $(\mY_t, \sum_{j=1}^kc_j\mB_{j,t})$ is a K-semistable weak log Fano pair.
\end{enumerate}
\end{definition}

In summary, we get several sets with the following inclusions:
$$\mG\subset \widehat{\mG}\subset \widehat{\mG}'\subset \widetilde{\mG}. $$
Usually it is not hard to prove the log boundedness of $\widehat{\mG}'$ under the log boundedness of $\widehat{\mG}$ (e.g. \cite[Theorem 1.7]{CLZ25}). Regarding the log boundedness of $\widetilde{\mG}$, we encounter the same difficulty as in Conjecture \ref{conj: bdd}, since $\widetilde{\mG}$ is obtained as the completion of $\widehat{\mG}'$ (see Step 2 of Definition \ref{def: dsm}) just as $\widehat{\mG}$ being obtained as the completion of $\mG$. Therefore, starting from the log boundedness of $\mG$, Conjecture \ref{conj: bdd} is the only obstacle to achieve the log boundedness of $\widetilde{\mG}$.

Suppose $\widetilde{\mG}$ is log bounded. Then replacing $\mG$ with $\widetilde{\mG}$, we can apply Definition \ref{def: K-complete} and Definition \ref{def: dsm} to extend $\widetilde{\mG}$ to $\widehat{\widetilde{\mG}}$ and $\widetilde{\widetilde{\mG}}$, just as we extend $\mG$ to $\widehat{\mG}$ and $\widetilde{\mG}$.

This boundedness (of $\widehat{\mG}$, $\widetilde{\mG}$, and $\widetilde{\widetilde{\mG}}$) seems tricky and we will treat it in future works. In this note, we just work with the assumption that $\widetilde{\mG}$ and $\widetilde{\widetilde{\mG}}$ are log bounded. Under the log boundedness, we have the following wall crossing property for K-semi(poly)stability.

\begin{theorem}{\rm (=\ Theorem \ref{thm: kps cd})}\label{mainthm: cd}
Under the assumption that $\widetilde{\mG}$ and $\widetilde{\widetilde{\mG}}$ are log bounded, there exists a finite chamber decomposition $P=\cup_i A_i$ such that $A_i$'s are all semi-algebraic sets and mutually disjoint; moreover, for any $(X, \sum_{j=1}^kD_j)\in \widehat{\mG}$ and any piece $A_i$, the K-semistability and K-polystability of $(X, \sum_{j=1}^kx_jD_j)$ do not depend on $(x_1,...,x_k)\in A_i$.
\end{theorem}

From the above theorem, each chamber $A_i$ gives one proper K-moduli space, thus there are only finitely many K-moduli spaces appearing as we change the coefficients. This gives a conditional positive answer to the finiteness question put at the beginning, under the assumption of the log boundedness of the set of K-semistable degenerations. The strategy towards Theorem \ref{mainthm: cd} is to prove the following two ingredients under the log boundedness of $\widetilde{\mG}$: (1) the number of K-semistable domains is finite, i.e., the following set is finite
$$\bigg\{\Kss\bigg(X, \sum_{j=1}^kD_j\bigg)_P\ \bigg|\  \textit{$\bigg(X, \sum_{j=1}^kD_j\bigg)\in \widehat{\mG}$}\bigg\},$$  
and (2) each K-semistable domain in the above set is semi-algebraic. Refer to Proposition \ref{prop: finite kss domain} and Proposition \ref{prop: semi-alg-2}.

\begin{remark}
We finally make a comment towards the above theorem. In \cite{LZ24b}, Theorem \ref{mainthm: cd} is worked out under the following volume condition: there exists a positive real number $\epsilon_0$ such that $\vol(-K_X-\sum_{j=1}^kx_jD_j)\geq \epsilon_0$ for any $(X, \sum_{j=1}^kD_j)\in \widehat{\mG}$ and any $(x_1,...,x_k)\in P$. The volume condition automatically implies that $\widehat{\mG}$ is log bounded (e.g. \cite[Section 7]{LZ24b}). However, we usually lose the volume condition as the coefficients approach the boundary of $P$.  The above theorem can be viewed as a conceptual generalization of the result proved in \cite{LZ24b}, which reduces the non-proportional wall crossing problem to the boundedness problem, i.e., Conjecture \ref{conj: bdd}.
\end{remark}

\noindent
\subsection*{Notation and Conventions}

\begin{itemize} 
\item A polytope in $\bR^k$ means a compact subset $P$ cut out by finitely many inequalities $H_i\geq 0$, where each $H_i$ determines a hyperplane in $\bR^k$; a face of the polytope $P$ is a subset of $P$ determined by the intersection of some hyperplanes $H_i=0$; 
\item We say the polytope $P\subset \bR^k$ is rational if all of its vertices admit rational coordinates;
\item For a polytope $P\subset \bR^k$, we use $P^\circ$  to denote the interior part of $P$ under the relative topology; 
\item For any subset $A\subset \bR^k$, we use $A(\bQ)$ to denote the set of rational points in $A$;
\item Given a polytope $P\subset \bR^k$, a finite polytope chamber decomposition of $P$ means a finite union $P=\cup_s P_s$, where each $P_s$ is a polytope and $P_i^\circ \cap P_j^\circ$ is empty for any different $i,j$; if each chamber $P_s$ is rational, we say it is a finite rational polytope chamber decomposition; 
\item Given a polytope $P\subset \bR^k$, a finite semi-algebraic chamber decomposition of $P$ means a finite union $P=\cup_s A_s$, where each chamber $A_s$ is semi-algebraic and any two chambers are disjoint;
\item When we mention a point on the base of a family of varieties, it always means a closed point; and a fiber of the family always means the fiber over a closed point;
\item A Weil divisor means a $\bZ$-divisor which is not necessarily irreducible;
\item A constructible stratification of a variety $T$ means decomposing $T$ into a finite disjoint union where each piece is locally closed.
\end{itemize}

\section{Preliminaries}\label{sec: 2}

We say $(V, B)$ is a \textit{couple} if $V$ is a normal projective variety and $B$ is an effective Weil divisor. We say $(V, B)$ is a \textit{log pair} if $V$ is a normal projective variety and $B$ is an effective $\bR$-divisor such that $K_V+B$ is $\bR$-Cartier. We say a log pair $(V, B)$ is \textit{log Fano} (resp. \textit{weak log Fano}) if $(V, B)$ admits klt singularities and $-K_V-B$ is ample (resp. big and nef).

We say $(X, \Delta)\to T$ is a family of couples if $X\to T$ is a projective morphism between finite type schemes, $\Delta$ is an effective divisor on $X$ such that every fiber $(X_t, \Delta_t)$ is a couple, and $\Delta\to T$ is a family of relative effective Mumford $\bZ$-divisor in the sense of Koll\'ar, see \cite[Definition 4.68]{Kollar23}. Given a family of log Fano pairs (resp. weak log Fano pairs) $(X, \Delta)\to T$ over a smooth base $T$, i.e., every fiber $(X_t, \Delta_t)$ is a log Fano pair (weak log Fano pair), we say it is an $\bR$-($\bQ$-)Gorenstein family if $-K_{X/T}-\Delta$ is $\bR$-($\bQ$-)Cartier.

For a comprehensive understanding on the K-stability of log Fano pairs with $\bQ$-coefficients, we refer to \cite{Xu21, Xu25}. For some basic understanding on the K-stability of log Fano pairs with $\bR$-coefficients, we refer to \cite{LZ24a}. Below we just briefly survey some elementary concepts and results that will be applied in this note.

\subsection{K-stability}

The following definition is inspired by \cite{Fuj19} and  \cite{Li17}.

\begin{definition} (Fujita-Li)\label{def: Fujita-Li}
For a given log Fano pair $(X, \Delta)$, we say it is \textit{K-semistable} if 
$\beta_{X, \Delta}(E):=A_{X, \Delta}(E)-S_{X, \Delta}(E)\geq 0$ for any prime divisor $E$ over $X$, where 
$$A_{X, \Delta}(E):=\ord_{E}\left(K_Y-f^*(K_X+\Delta)\right)+1, $$
$$S_{X, \Delta}(E):=\frac{1}{\vol(-K_X-\Delta)}\int_0^{\infty}\vol(-f^*(K_X+\Delta)-tE){\rm d} t .$$
Here $f: Y\to (X, \Delta)$ is a proper log resolution such that $E$ is a prime divisor on $Y$. The quantity $\beta_{X, \Delta}(E)$ is also called the Fujita-Li invariant of $(X, \Delta)$ with respect to $E$. For a given weak log Fano pair, we say it is K-semistable if its anti-canonical model (which is a log Fano pair) is K-semistable.
\end{definition}

\begin{definition}\label{def: kps}
Let $(X, \Delta)$ be a K-semistable log Fano pair. We say it is  \textit{K-polystable} if any special test configuration of $(X, \Delta; -K_X-\Delta)$ with a K-semistable central fiber is of product type.
\end{definition}

By \cite{LWX21}, a K-semistable log Fano pair with $\bQ$-coefficients always admits a K-polystable degeneration via a special test configuration. The same result applies to K-semistable log Fano pairs with $\bR$-coefficients by \cite{LZ24a}.

Before Definition \ref{def: Fujita-Li}, K-semistability of a log Fano pair is tested by the non-negativity of the generalized Futaki invariants of special test configurations (see \cite{LX14, LZ24a}). The generalized Futaki invariant can be reformulated by the CM-line bundle.

Let $Q\subset [0,1]^k$ be a rational polytope and $\pi: (\mX, \sum_{j=1}^k\mD_j)\to T$ is a family of couples of relative dimension $d$ over a smooth base $T$. Suppose $\pi_{\bm x}: (\mX, \sum_{j=1}^kx_j\mD_j)\to T$ is an $\bR$-Gorenstein family of log Fano pairs for any $\bm x:=(x_1,...,x_k)\in Q$. By the work of Mumford-Knudsen \cite{KM76}, for any $\bm x\in Q(\bQ)$, there exist (uniquely determined) $\bQ$-line bundles $\lambda_i(\bm x) \ (i=0,1,...,d+1)$ on $T$ such that we have the following expansion for all sufficiently divisible $m\in \bN$:
$$\det \pi_{\bm x,*}(\mL(\bm x)^{\otimes m})= \lambda_{d+1}(\bm x)^{\binom{m}{d+1}}\otimes\lambda_d(\bm x)^{\binom{m}{d}}\otimes\cdot\cdot\cdot\otimes\lambda_1(\bm x)^{\binom{m}{1}}\otimes\lambda_0(\bm x).$$
For convenience, we write $\mD(\bm x):=\sum_{j=1}^kx_j\mD_j$ and $\mL(\bm x):=-K_{\mX/T}-\sum_{j=1}^kx_j\mD_j$.

\begin{definition}{\rm (\cite[Definition 9.29]{Xu25})}\label{def: Q-CM}
For any given rational vector $\bm x\in Q(\bQ)$, the CM-line bundle for the family $\pi_{\bm x}: (\mX,\mD(\bm x);\mL(\bm x))\to T$ is a $\bQ$-line bundle on $T$ defined as $\lambda_{\CM}(\bm x):=-\lambda_{d+1}(\bm x)$, which depends on $\bm x\in Q(\bQ)$.
\end{definition}

The above definition only applies to rational $\bm x\in Q(\bQ)$. It is generalized to arbitrary $\bm x\in Q$ as follows.

\begin{proposition-definition}{\rm (\cite[Section 8.1]{LZ24b})}\label{def: R-CM}
Notation as in Definition \ref{def: Q-CM}, there exist finitely many $\bQ$-line bundles $M_q$'s on $T$ and finitely many $\bQ$-polynomials $g_q(x_1,...,x_k)$'s such that
$$\lambda_{\CM}(\bm x)= \sum_q g_q(x_1,...,x_k) \cdot M_q $$
for any $\bm x=Q(\bQ)$. Moreover, if $\pi: (\mX, \sum_{j=1}^k\mD_j)\to T$ admits a $G$-equivariant action for some algebraic group $G$, then $M_q$'s can be chosen $G$-equivariant. By the above formulation, one can naturally extend the definition of $\CM$-line bundle from $\bm x\in Q(\bQ)$ to arbitrary $\bm x\in Q$.
\end{proposition-definition}

We have the following well-known property of the CM line bundle.

\begin{proposition}{\rm (\cite{PT06})}\label{prop: Q-PT}
Notation as above. Suppose for any given $\bm x\in Q(\bQ)$,  $\pi_{\bm x}: (\mX,\mD(\bm x);\mL(\bm x))\to T$ is a special test configuration {\rm (}$T=\bA^1$ in this case{\rm )} of a log Fano pair $(X, \sum_{j=1}^kx_jD_j)$. Write $L(\bm x):=-K_X-\sum_{j=1}^k x_jD_j$. Then $\lambda_{\CM}(\bm x)$ is a $\bG_m$-linearized $\bQ$-line bundle on the base $\bA^1$ and we have
$$\Fut(\mX,\mD(\bm x);\mL(\bm x))=\frac{1}{(d+1)L(\bm x)^d}\cdot \wt(\lambda_{\CM}(\bm x)).$$
\end{proposition}

Under the condition of Proposition \ref{prop: Q-PT}, we know that 
$\pi_{\bm x}: (\mX,\mD(\bm x);\mL(\bm x))\to T$ is a special test configuration for arbitrary $\bm x\in Q$. Note that 
$$\wt(\lambda_{\CM}(\bm x))=\sum_q g_q(\bm x)\cdot\wt(M_q)$$
for some $\bG_m$-linearized line bundles $M_q$'s and some polynomials $g_q(\bm x)$'s (see Proposition-Definition \ref{def: R-CM}). We thus see that the formulation of the generalized Futaki invariant in Proposition \ref{prop: Q-PT} also applies to arbitrary $\bm x\in Q$ by the continuity of both sides.

\subsection{Boundedness of algebraic varieties}

Let $\mP$ be a set of projective varieties of dimension $d$. We say $\mP$ is bounded if there exists a projective morphism $ \mY\to T$ between schemes of finite type such that for any $X\in \mP$, there exists a closed point $t\in T$ such that $X\cong \mY_t$. Let $\mP'$ be a set of  couples of dimension $d$. We say 
$\mP'$ is log bounded if there exist a projective morphism $\mY\to T$ between schemes of finite type and a Weil divisor $\mD$ on $\mY$ such that for any $(X, \Delta)\in \mP$, there exists a closed point $t\in T$ such that $(X, \Delta)\cong (\mY_t, \mD_t)$. 

\section{Log Fano domains and K-semistable domains}\label{sec: 3}
\subsection{Constructible log Fano domains}
Let $A\subset \bR^k$ be a subset. Given a couple $(X, \sum_{j=1}^k D_j)\in \widetilde{\mG}$, recall the following definition of log Fano domain (resp. weak log Fano domain) restricted to the set $A$:
$$\LF\bigg(X, \sum_{j=1}^kD_j\bigg)_A:=\bigg\{(x_1,...,x_k)\in A\ \bigg|\ \text{$(X, \sum_{j=1}^kx_jD_j)$ is a log Fano pair}\bigg\}; $$
$$\bigg({\rm resp.}\ \ \WLF\bigg(X, \sum_{j=1}^kD_j\bigg)_A:=\bigg\{(x_1,...,x_k)\in A\ \bigg|\ \text{$(X, \sum_{j=1}^kx_jD_j)$ is a weak log Fano pair}\bigg\}\bigg)$$

For a couple $(X, \sum_{j=1}^k D_j)\in \widetilde{\mG}$, we know that the closure of $\LF(X, \sum_{j=1}^kD_j)_P$  is a rational polytope (e.g. \cite[Proposition 3.3]{LZ24b}). If we write $Q=\overline{\LF(X, \sum_{j=1}^kD_j)_P}$ and let $\{F_l\}_{l=1}^m$ be the set of all proper faces of $Q$, then we have the following decomposition:
$$\LF(X, \sum_{j=1}^kD_j)_P= Q^\circ\ \cup\ \cup_{i=1}^s F^\circ_{r_i},  $$
where $\{r_1, r_2,..., r_s\}$ is a subset of $\{1,2,...,m\}$. The weak log Fano domain satisfies the same property.

\begin{proposition}\label{prop: first stra}
Suppose there exists a family of couples over a finite type scheme $\pi: (\mX, \sum_{j=1}^k\mD_j)\to T$ such that $\widetilde{\mG}$ is contained in the set of fibers of $\pi$. Then there exists a constructible stratification $T=\amalg_i T_i$ satisfying the following conditions:
\begin{enumerate}
\item each $T_i$ is smooth;
\item the induced family (obtained via base-change) $\pi_i: (\mX^{(i)}, \sum_{j=1}^k \mD_j^{(i)})\to T_i$ is flat; 
\item the fibers of $(\mX^{(i)}, \sum_{j=1}^k \mD_j^{(i)})\to T_i$ admit the same log Fano domain (resp. weak log Fano domain) restricted to $P$ for any index $i$.
\end{enumerate}
\end{proposition}

\begin{proof}
By \cite[Proposition 4.1]{LZ24b}, up to replacing $T$ with another finite type scheme, we may assume that for any $t\in T$, the fiber $(\mX_t, \sum_{j=1}^k c_j\mD_{j,t})$ is log Fano or weak log Fano for some $(c_1,...,c_k)\in P$. We are done by applying \cite[Proposition 3.5]{LZ24b}.
\end{proof}

\begin{proposition}\label{prop: uni fano}
There exists a finite rational polytope chamber decomposition $P=\cup_s P_s$ such that for any face $F$ of any chamber $P_s$, we have the following property:
for any $(X,\sum_{j=1}^kD_j)\in \widetilde{\mG}$, if $(X, \sum_{j=1}^kc_jD_j)$ is a log Fano pair (resp. weak log Fano pair) for some vector $(c_1,...,c_k)\in F^\circ$, then $(X,\sum_{j=1}^kx_jD_j)$ is a log Fano pair (resp. weak log Fano pair) for any vector $(x_1,...,x_k)\in F^\circ$.
\end{proposition}

\begin{proof}
We only deal with the log Fano case, as the weak log Fano case can be proved similarly. Note that $\overline{\LF(X, \sum_{j=1}^kD_j)_P}$ is a rational polytope for any $(X, \sum_{j=1}^kD_j)$. Put $Q:=\overline{\LF(X, \sum_{j=1}^kD_j)_P}$, and let $\{F_l\}_{l=1}^m$ be the set of all proper faces of $Q$, then we have the following decomposition
$$\LF(X, \sum_{j=1}^kD_j)_P=Q^\circ\ \cup\ \cup_{i=1}^s F^\circ_{r_i},  $$
where $\{r_1, r_2,..., r_s\}$ is a subset of $\{1,2,...,m\}$. By the constructible property of log Fano domains as presented in Proposition \ref{prop: first stra}, there exists a desired finite rational polytope chamber decomposition of $P$ (by refining the finite log Fano domains).
\end{proof}

\subsection{Finiteness of K-semistable domains}

Let $A\subset \bR^k$ be a subset. Given a couple $(X, \sum_{j=1}^k D_j)\in \widetilde{\mG}$, recall the following definition of K-semistable domain restricted to the set $A$:
$$\Kss(X, \sum_{j=1}^kD_j)_A:=\{(x_1,...,x_k)\in A\ |\ \text{$(X, \sum_{j=1}^kx_jD_j)$ is a K-semistable log Fano pair}\}.$$

\begin{proposition}\label{prop: finite kss domain}
Under the assumption that the set $\widetilde{\mG}$ is log bounded, the following set of K-semistable domains restricted to $P$ is finite:
$$\bigg\{\Kss\bigg(X, \sum_{j=1}^kD_j\bigg)_P\ \bigg| \ \text{$(X, \sum_{j=1}^kD_j)\in \widetilde{\mG}$}\bigg\}.$$
In particular, the set $\{\Kss(X, \sum_{j=1}^kD_j)_P\ | \ \text{$(X, \sum_{j=1}^kD_j)\in \widehat{\mG}$}\}$ is finite.
\end{proposition}

\begin{proof}
Given the log boundedness of $\widetilde{\mG}$, there exists a family $\pi: (\mX, \sum_{j=1}^k\mD_j)\to T$ as in Proposition \ref{prop: first stra} and there is a constructible stratification of $T$ such that all the conditions listed in Proposition \ref{prop: first stra} are satisfied. The finiteness property of K-semistable domains then follows from \cite[Theorem 5.4]{LZ24b}.
\end{proof}

\section{Hilbert schemes}\label{sec: 5}

Recall that given a family of log Fano (resp. weak log Fano) $\bQ$-($\bR$-)pairs $(Y, \Delta)\to B$ over a smooth base $B$,  we say it is a $\bQ$-($\bR$-) Gorenstein family if $K_Y+\Delta$ is $\bQ$-($\bR$-)Cartier. 
In this section,  we aim to find some families with certain \textit{uniform $\bR$-Gorenstein property} and certain \textit{universality} to parametrize the couples in $\widetilde{\mG}$. For any subset $A\subset P$, we define the following subset of $\widetilde{\mG}$ with respect to $A$:
$$\widetilde{\mG}_A:=\{(X, \sum_{j=1}^kD_j)\in \widetilde{\mG}\ |\ \text{$(X, \sum_{j=1}^kx_jD_j)$ is log Fano  for any $(x_1,...,x_k)\in A$}\}. $$

From Section \ref{sec: 3}, it follows that $\widetilde{\mG}$ is parametrized by finitely many families of couples $\pi_i: (\mX^{(i)}, \sum_{j=1}^k \mD_j^{(i)})\to T_i$, where each $T_i$ is irreducible and smooth, and for each family the fibers admit the same log Fano domain or weak log Fano domain (see Proposition \ref{prop: first stra}). However, the parametrization by $\{\pi_i\}_i$ is not good enough for latter use, we aim to refine each $T_i$ to get finer parametrization. 

\begin{lemma}\label{lem: uni R-gorenstein}
We focus on one family $\pi_i: (\mX^{(i)}, \sum_{j=1}^k \mD_j^{(i)})\to T_i$ and write $Q_i$ to be the closure of the log Fano domain (resp. the weak log Fano domain) of the fibers. Then there exists an open subset $U_i\subset T_i$ such that $\pi_{i, U_i}: (\mX^{(i)}, \sum_{j=1}^k x_j\mD_j^{(i)})\times_{T_i} U_i\to U_i$ is an $\bR$-Gorenstein family of log Fano pairs (resp. weak log Fano pairs) for any $(x_1,...,x_k)\in Q_i^\circ$.
\end{lemma}

\begin{proof}
We may assume $Q_i$ is the closure of the log Fano domain, since the weak log Fano domain case can be proved similarly.
It suffices to show that there exists an open subset $U_i\subset T_i$ such that $K_{\mX^{(i)}}+\sum_{j=1}^kx_j\mD_j^{(i)}$ is $\bR$-Cartier over $U_i$ for any $(x_1,...,x_k)\in Q_i^{\circ}$. Note that for any rational $(a_1,...,a_k)\in Q_i^{\circ}$, the family $(\mX^{(i)}, \sum_{j=1}^k a_j\mD_j^{(i)})\to T_i$ is a family of log Fano pairs, thus there exists an open subset $V\subset T_i$ such that $K_{\mX^{(i)}}+\sum_{j=1}^ka_j\mD_j^{(i)}$ is $\bQ$-Cartier over $V$. Inside $Q_i^\circ$, we select finitely many rational vectors $\{(a_{l1},...,a_{lk})\}_l$ that span a rational polytope of the same dimension as $Q_i$, then there exist corresponding open subsets $V_l\subset T_i$ such that $K_{\mX^{(i)}}+\sum_{j=1}^ka_{lj}\mD_j^{(i)}$ is $\bQ$-Cartier over $V_l$. Taking $U_i=\cap_l V_l$, the desired property follows.
\end{proof}

By Lemma \ref{lem: uni R-gorenstein}, up to stratifying each $T_i$ in a constructible way, we may assume that for each parametrizing family $\pi_i: (\mX^{(i)}, \sum_{j=1}^k \mD_j^{(i)})\to T_i$, not only the fibers share the same closure of the log Fano domain (resp. weak log Fano domain) $Q_i$,  but also $\pi_{i}: (\mX^{(i)}, \sum_{j=1}^k x_j\mD_j^{(i)})\to T_i$ is an $\bR$-Gorenstein family of log Fano pairs (resp. weak log Fano pairs) for any $(x_1,...,x_k)\in Q_i^\circ$. Recalling that in Proposition \ref{prop: uni fano}, by refining the finite rational polytopes $\{Q_i\}$, we get a finite rational polytope chamber decomposition $P=\cup_s P_s$ such that, for any $(X, \sum_{j=1}^kD_{j})\in \widetilde{\mG}$ and any face $F$ of any chamber $P_s$, if $(X, \sum_{j=1}^kc_jD_{j})$ is a log Fano pair (resp. weak log Fano pair) for some vector $(c_1,...,c_k)\in F^\circ$, then $(X, \sum_{j=1}^kx_jD_{j})$ is a log Fano pair (resp. weak log Fano pair) for any vector $(x_1,...,x_k)\in F^\circ$. Therefore, it follows 
$$\widetilde{\mG}_{F^\circ}:=\{(X, \sum_{j=1}^kD_j)\in \widetilde{\mG}\ |\ \text{$(X, \sum_{j=1}^kx_jD_j)$ is log Fano  for any (or some) $(x_1,...,x_k)\in F^\circ$}\}. $$
Fix a face $F$ of some chamber $P_s$. Next we will focus on the set $\widetilde{\mG}_{F^\circ}$, and we thus select those families $\pi_i: (\mX^{(i)}, \sum_{j=1}^k \mD_j^{(i)})\to T_i$ whose fibers share the same closure of log Fano domain $Q_i$ with $F^\circ\subset Q_i$. We denote these families by $f_i: (\mY^{(i)}, \sum_{j=1}^k\mB_j^{(i)})\to S_i$ ($S_i$ being irreducible and smooth) and the set $\widetilde{\mG}_{F^\circ}$ is thus contained in the set of fibers of those $f_i$'s. In the lemma below, we aim to find a Hilbert scheme to parametrize the fibers of $f_i$ to achieve the universality.

\begin{lemma}\label{lem: hilbert-para}
For each family $f_i: (\mY^{(i)}, \sum_{j=1}^k\mB_j^{(i)})\to S_i$, there exist a Hilbert scheme $W_i$ and a universal family $\phi_i: (\bm Y^{(i)}, \sum_{j=1}^k \bm B^{(i)}_j)\to W_{i}$ such that the following conditions are satisfied:
\begin{enumerate}
\item The set of fibers of $f_i$ is contained in the set of fibers of $\phi_i$;
\item The base $W_i$ admits a $\PGL$-group action;
\item Let $\{\bm a_r:=(a_{r1},...,a_{rk})\}_r$ be the collection of vertices of $F$ (which are all rational vectors), then for each $r$, the following sheaf
$$\omega_{\bm Y^{(i)}/W_{i}}^{[-n_r]}(-n_r\cdot \sum_{j=1}^k a_{rj}\bm B^{(i)}_j)$$
is invertible for sufficiently divisible positive integer $n_r$;
\item $({\bm Y}_t^{(i)}, \sum_{j=1}^k x_j\bm B^{(i)}_{j,t})$ is log Fano for all $t\in W_i$ and $(x_1,...,x_k)\in F^\circ$.
\end{enumerate}
\end{lemma}

\begin{proof}

By the construction, we know that $(\mY^{(i)}, \sum_{j=1}^kx_j\mB^{(i)}_j)\to S_i$ is an $\bR$-Gorenstein family of log Fano pairs for any $(x_1,...,x_k)\in F^\circ$, thus the Hilbert function 
$$m\  \text{(sufficiently divisible)}\mapsto \chi\left(-m(K_{\mY_t^{(i)}}+\sum_{j=1}^k x_j\mB^{(i)}_{j, t})\right)$$
does not depend on $t\in S_i$ for any fixed rational vector $(x_1,...,x_k)\in  F^\circ(\bQ)$.

Fix a rational vector $(c_1,...,c_k)\in F^\circ(\bQ)$ and choose a sufficiently divisible positive integer $N$ such that $-N(K_{\mY^{(i)}/S_i}+\sum_{j=1}^{k} c_j\mB^{(i)}_j)$ is very ample over $S_i$. Write
$$h(m):= \chi\left(-mN(K_{\mY_t^{(i)}}+\sum_{j=1}^k c_j\mB^{(i)}_{j, t})\right)\ \  \text{and}\ \ M:=h(1)-1.$$
Let $\Hilb_{h}(\bP^{M})$ be the Hilbert scheme parametrizing closed subschemes of $\bP^{M}$ with Hilbert polynomial $h(m)$. Let $V_i\subset \Hilb_{h}(\bP^{M})$ be the open subscheme parametrizing normal and Cohen-Macaulay varieties, and denote the universal family by $\kY_i\to V_i$ as in the following diagram
\begin{center}
	\begin{tikzcd}[column sep = 2em, row sep = 2em]
	 \kY_i \arrow[rd,""] \arrow[rr,"l_i"]&& \bP_{V_i}^M \arrow[ld,""]\\
	 &V_i&.
	\end{tikzcd}
\end{center}
We see that $V_i$ admits a $\PGL(M+1)$-action. Write $\kL_i:=l_i^*\mO_{\bP_{V_i}^M}(1)$.

Let $\mG_i$ be the set of fibers of $(\mY^{(i)}, \sum_{j=1}^{k} \mB^{(i)}_j)\to S_i$. For each $(X, \sum_{j=1}^k D_j)\in \mG_i$, we denote 
$$d_j(X, \sum_{j=1}^k D_j):=(-N(K_X+\sum_{j=1}^k c_jD_j))^{d-1}\cdot D_j.$$ 
Note that $(\mY^{(i)}, \sum_{j=1}^kc_j\mB^{(i)}_j)\to S_i$ is a $\bQ$-Gorenstein family of log Fano pairs. By Lemma \ref{lem: constant degree}, we know that $d_j(X, \sum_{j=1}^k D_j)$ does not depend on $(X, \sum_{j=1}^k D_j)\in \mG_i$, and we thus write $d_j:=d_j(X, \sum_{j=1}^k D_j)$.
For the vector of degrees $(d_1,...,d_k)$, there exists a Chow scheme 
$$\Chow_{d_1,...,d_k}(\kY_i/V_i, \kL_i)$$
parametrizing all codim one cycles of the form $(Z_1,...,Z_k)$ in all fibers of $\kY_i\to V_i$ with degrees $(d_1,...,d_k)$. More precisely,  $Z_1,...,Z_k$ are codim one cycles on some fiber of $\kY_i\to V_i$, say $\kY_{i,t}$, and the degree of $Z_j$ with respect to the polarization $\kL_{i,t}$ is $d_j$. We denote the universal family over
 $\Chow_{d_1,...,d_k}(\kY_i/V_i, \kL_i)$ by 
\begin{center}
	\begin{tikzcd}[column sep = 2em, row sep = 2em]
	 (\kY_i\times_{V_i} \Chow_{d_1,...,d_k}(\kY_i/V_i, \kL_i), \kZ_{i1}+..+\kZ_{ik}) \arrow[rd,""] \arrow[rr,""]&& \Chow_{d_1,...,d_k}(\kY_i/V_i, \kL_i) \arrow[ld,""]\\
	 &V_i&,
	\end{tikzcd}
\end{center}
where for each $j$, $\psi_j$ in the following diagram
\begin{center}
	\begin{tikzcd}[column sep = 2em, row sep = 2em]
	\kZ_{ij} \arrow[rd,"\psi_j"] \arrow[rr,""]&& \kY_i\times_{V_i} \Chow_{d_1,...,d_k}(\kY_i/V_i, \kL_i) \arrow[ld,""]\\
	 &\Chow_{d_1,...,d_k}(\kY_i/V_i, \kL_i)/V_i&
	\end{tikzcd}
\end{center}
is the universal family of codimension one cycles in fibers of $\kY_i\to V_i$ with degree $d_j$. 
By \cite[Theorem 7.3]{Kollar23}, there exists a separated $V_i$-scheme $W_{i1}$ of finite type parametrizing K-flat divisors $(D_1,...,D_k)$ with degrees $(d_1,...,d_k)$. 
Let
$$(\bm X^{(i)}, \sum_{j=1}^k \bm D^{(i)}_j)\to W_{i1}$$ 
be the corresponding universal family. 
By the construction, it follows that the set $\mG_i$ is contained in the set of fibers of $(\bm X^{(i)}, \sum_{j=1}^k \bm D^{(i)}_j)\to W_{i1}$, and $W_{i1}$ admits a $\PGL$-group action.

Note that $\{\bm a_r:=(a_{r1},...,a_{rk})\}_r$ is the collection of vertices of $F$. For each $r$, we choose a sufficiently divisible positive integer $n_r$ such that $-n_r(K_X+\sum_{j=1}^k a_{rj}D_j)$ is a line bundle for any $(X, \sum_{j=1}^k D_j)\in \mG_i$. Denote 
$$L_r:=\omega_{\bm X^{(i)}/W_{i1}}^{[-n_r]}(-n_r\cdot \sum_{j=1}^k a_{rj}\bm D^{(i)}_j).$$
By \cite[Corollary 3.30]{Kollar23}, there exists a locally closed subscheme $W_{2i}\subset W_{1i}$ such that a map $W\to W_{1i}$ factors through $W_{2i}$ if and only if 
$$(L_r)_W:= \omega_{\bm X^{(i)}_W/W}^{[-n_r]}(-n_r\cdot \sum_{j=1}^k a_{rj}(\bm D^{(i)}_j)_W)$$
is invertible for any index $r$, where $(\bm X^{(i)}_W, \sum_{j=1}^k (\bm D^{(i)}_j)_W)\to W$
is obtained by the base change along $W\to W_{1i}$. We thus see that the set $\mG_i$ is contained in the set of fibers of the following family
$$(\bm X^{(i)}_{W_{i2}}, \sum_{j=1}^k (\bm D^{(i)}_j)_{W_{i2}})\to W_{i2}, $$
and $W_{i2}$ also admits a $\PGL$-group action (induced by that of $W_{i1}$). By the construction,  
$$\omega_{\bm X_{W_{i2}}^{(i)}/W_{i2}}^{[-n_r]}(-n_r\cdot \sum_{j=1}^k a_{rj}(\bm D^{(i)}_j)_{W_{i2}})$$ 
is invertible for all $r$. Consider the following set
$$W_i:=\{t\in W_{i2}\ |\ \text{$(\bm X^{(i)}_t, \sum_{j=1}^k b_j\bm D^{(i)}_{j,t})$ is log Fano for any $(b_1,...,b_k)\in F^\circ$}\}. $$
By \cite[Proof of Proposition 3.5]{LZ24b}, it follows that $W_i$ is a subscheme of $W_{i2}$ and $W_i$ still admits a $\PGL$-group action (induced by that of $W_{i1}$). By a base-change, we obtain a family of couples
$$\phi_i: (\bm X^{(i)}_{W_i}, \sum_{j=1}^k (\bm D^{(i)}_j)_{W_i})\to W_i,$$ 
which satisfies all the desired requirements in Lemma \ref{lem: hilbert-para}.
\end{proof}

The following lemma is applied in the above proof.

\begin{lemma}\label{lem: constant degree}
Given a $\bQ$-Gorenstein family of log Fano pairs of dimension $d$, denoted by $h: (\mY, \sum_{j=1}^kc_j\mB_j)\to S$, over an irreducible smooth base $S$, where $\mB_j$ is a Weil divisor on $\mY$ for each $j$. 
Then the intersection number $(-K_{\mY_t}-\sum_{j=1}^kc_j\mB_{j,t})^{d-1}\cdot \mB_{j,t}$ is independent of $t\in S$. 
\end{lemma}

\begin{proof}
As the statement is local, we may argue by shrinking the base $S$. Let $g: (\mY', \sum_{j=1}^kc_j\mB'_j)\to (\mY, \sum_{j=1}^kc_j\mB_j)$ be a small $\bQ$-factorial modification, then the composition $h\circ g: (\mY', \sum_{j=1}^kc_j\mB'_j)\to S$ is a $\bQ$-Gorenstein family of weak log Fano pairs (see the second paragraph of the proof of Proposition \ref{prop: uni gorenstein}). By \cite[Lemma 10.58]{Kollar23}, the morphism $h\circ g: \mY'\to S$ is flat. 

Assume $n=\dim S$. Taking $n$ sufficiently general hypersurfaces passing through $t\in S$, denoted by $H_1,...,H_n$. Up to shrinking $S$, we may assume these $H_i$'s intersect at the unique point $t\in S$. By the $\bQ$-factoriality of $\mY'$ and the flatness of $\mY'\to S$, we know that the following intersection number is independent of $t\in S$:
$$g^*(-K_{\mY}-\sum_{j=1}^kc_j\mB_{j})^{d-1}\cdot g^*h^*H_1\cdot...\cdot g^*h^*H_n\cdot \mB_j'=(-K_{\mY'_t}-\sum_{j=1}^kc_j\mB'_{j,t})^{d-1}\cdot \mB'_{j,t}. $$
By projection formula, the left hand side of the above formula is equal to the following quantity:
$$(-K_{\mY}-\sum_{j=1}^kc_j\mB_{j})^{d-1}\cdot h^*H_1\cdot...\cdot h^*H_n\cdot \mB_j= (-K_{\mY_t}-\sum_{j=1}^kc_j\mB_{j,t})^{d-1}\cdot \mB_{j,t}.$$
The proof is complete.
\end{proof}

\begin{lemma}\label{lem: university}
Given a family of couples $(\mY, \sum_{j=1}^k\mB_j)\to S$ over the smooth base $S:=\bA^1$ such that all of its fibers are contained in $\widetilde{\mG}$ and $(\mY,\sum_{j=1}^kx_j\mB_j)\to S$ is an $\bR$-Gorenstein family of log Fano pairs for any $(x_1,...,x_k)\in F^\circ$ (see the paragraph above Lemma \ref{lem: hilbert-para} for $F^\circ$). Then there exists an index $i$ such that $(\mY, \sum_{j=1}^k\mB_j)\to S$ is  the pull-back along a morphism $S\to W_i$, where $\phi_i: (\bm Y^{(i)}, \sum_{j=1}^k\bm D^{(i)}_j)\to W_i$ is the family constructed in Lemma \ref{lem: hilbert-para}.
\end{lemma}

\begin{proof}
First note that the following Hilbert function $(\spadesuit)$
$$m\  \text{(sufficiently divisible)}\mapsto \chi\left(-m(K_{\mY_t}+\sum_{j=1}^k x_j\mB_{j, t})\right)$$
does not depend on $t\in S$ for any fixed rational vector $(x_1,...,x_k)\in  F^\circ(\bQ)$. Observe that there also exists an index $i$ such that the family $f_i: (\mY^{(i)}, \sum_{j=1}^k\mB^{(i)}_j)\to S_i$ (see the paragraph above Lemma \ref{lem: hilbert-para} for $f_i$) admits the same Hilbert function as $(\spadesuit)$. By the Hilbert scheme construction in the proof of Lemma \ref{lem: hilbert-para}, there exists a morphism $S\to W_i$ (by the universality) such that the family $(\mY, \sum_{j=1}^k\mB_j)\to S$ is the pull-back of $\phi_i$ along $S\to W_i$, where $\phi_i: (\bm Y^{(i)}, \sum_{j=1}^k\bm D^{(i)}_j)\to W_i$ is the family constructed in Lemma \ref{lem: hilbert-para}.
\end{proof}

\

By now, we construct finitely many universal families $\phi_i: (\bm Y^{(i)}, \sum_{j=1}^k\bm D^{(i)}_j)\to W_i$ parametrizing the set $\widetilde{\mG}_{F^\circ}$. Here is a map to understand this construction.

\textit{Step 1.} A priori, we know that the set $\widetilde{\mG}$ is parametrized by finitely many family of couples $\pi_i: (\mX^{(i)}, \sum_{j=1}^k \mD_j^{(i)})\to T_i$, where each $T_i$ is smooth and the fibers of $\pi_i$ share the same closure of the log Fano domain or weak log Fano domain, denoted by $Q_i$.

\textit{Step 2.} The families $\pi_i$'s are not good enough for applications, so we further stratify each $T_i$ such that for each family (after stratification) $\pi_i: (\mX^{(i)}, \sum_{j=1}^k \mD_j^{(i)})\to T_i$, we even have the following stronger condition: $(\mX^{(i)}, \sum_{j=1}^k x_j\mD_j^{(i)})\to T_i$ is an $\bR$-Gorenstein family of log Fano pairs (resp. weak log Fano pairs) for any $(x_1,...,x_k)\in Q_i^\circ$ (see Lemma \ref{lem: uni R-gorenstein}).

\textit{Step 3.} As we will mainly focus on a face $F$ of a chamber $P_s$, where $P=\cup_s P_s$ is the rational polytope chamber decomposition in Proposition \ref{prop: uni fano}, we select those families $\pi_i$ in Step 2 satisfying $F^\circ\subset Q_i$, and we denote them by $f_i$'s (see the paragraph above Lemma \ref{lem: hilbert-para}). The families $f_i$'s parametrize the set $\widetilde{\mG}_{F^\circ}$ by the definition of $\widetilde{\mG}_{F^\circ}$. Lemma \ref{lem: hilbert-para} is then designed to pack each family $f_i$ into a universal family. The universality presented in Lemma \ref{lem: university} will be used later (e.g. Proposition \ref{prop: semi-alg-2}).

\begin{proposition}\label{prop: uni gorenstein}
Under the assumption that $\widetilde{\mG}$ is log bounded, there exists a finite rational polytope chamber decomposition $P=\cup_s P_s$ satisfying the following: for any face $F$ of any chamber $P_s$, if $h: (\mY, \sum_{j=1}^kc_j\mB_j)\to S$ is an $\bR$-Gorenstein family of K-semistable log Fano pairs over a smooth base $S$ for some $(c_1,...,c_k)\in F^\circ$ with $(\mY_t, \sum_{j=1}^kB_{j,t})\in \widetilde{\mG}$ 
for any $t\in S$ and $(\mY_t, \sum_{j=1}^kB_{j,t})\in \widehat{\mG}$ for general $t\in S$, then $(\mY, \sum_{j=1}^kx_j\mB_j)\to S$ is an $\bR$-Gorenstein family of log Fano pairs for any $(x_1,...,x_k)\in F^\circ$.
\end{proposition}

\begin{proof}
Let $P=\cup_s P_s$ be the finite rational polytope chamber decomposition in  Proposition \ref{prop: uni fano}, then we have the following property: for any $(X, \sum_{j=1}^kD_j)\in \widetilde{\mG}$ and any face $F$ of any chamber $P_s$, if $(X, \sum_{j=1}^k c_jD_j)$ is a log Fano pair (resp. weak log Fano pair) for some $(c_1,...,c_k)\in F^\circ$, then $(X, \sum_{j=1}^k x_jD_j)$ is a log Fano pair (resp. weak log Fano pair) for any $(x_1,...,x_k)\in F^\circ$.

Given the $\bR$-Gorenstein family of K-semistable log Fano pairs $h: (\mY, \sum_{j=1}^kc_j\mB_j)\to S$ for some $(c_1,...,c_k)\in F^\circ$ as in the statement of the proposition, we consider the following small $\bQ$-factorial modification (\cite[Corollary 1.4.3]{BCHM10}):
$$g: (\mY', \sum_{j=1}^kc_j\mB'_j)\to (\mY, \sum_{j=1}^kc_j\mB_j) /S.$$
We claim that $(\mY', \sum_{j=1}^kc_j\mB'_j)\to S$ is an $\bR$-Gorenstein family of K-semistable weak log Fano pairs. 
To see this, we first note the following formula:
$$K_{\mY'/S}+\sum_{j=1}^kc_j\mB_j' =g^*(K_{\mY/S}+\sum_{j=1}^kc_j\mB_j).$$
Write $n=\dim S$. Taking $n$ sufficiently general hypersurfaces passing through $t\in S$, denoted by $H_1,...,H_n$,  we have
$$K_{\mY'/S}+\sum_{j=1}^kc_j\mB_j' +\sum_{r=1}^ng^*h^*H_r=g^*(K_{\mY/S}+\sum_{j=1}^kc_j\mB_j+\sum_{r=1}^n h^*H_r) .$$ 
Observe that both sides are dlt by the generality of $H_r$'s. Applying adjunction we have the following formula:
$$K_{\mY'_t}+\sum_{j=1}^rc_j\mB'_{j,t}=\phi^*(K_{\mY_t}+\sum_{j=1}^kc_j\mB_{j,t}). $$
This implies that $(\mY'_t, \sum_{j=1}^kc_j\mB'_{j,t})$ is a weak log Fano pair and the claim is concluded.
On the other hand, $(\mY'_t, \sum_{j=1}^k\mB'_{j,t})$ is a small modification of $(\mY_t, \sum_{j=1}^k\mB_{j,t})$ for general $t\in T$ (by dimensional reason), which implies  $(\mY'_t, \sum_{j=1}^k\mB'_{j,t})\in\widehat{\mG}' \subset \widetilde{\mG}$ for general $t\in S$ and thus for every $t\in S$
(see the description of the sets $\widehat{\mG}'$ and $\widetilde{\mG}$ in Definition \ref{def: dsm}). 

Recalling the choice of the chamber decomposition in the first paragraph, we deduce that $(\mY', \sum_{j=1}^kx_j\mB'_j)\to S$ is an $\bR$-Gorenstein family of weak log Fano pairs for any $(x_1,...,x_k)\in F^\circ$. By the same argument of \cite[Proof of Proposition 7.3]{LZ24b}, we know that $(\mY, \sum_{j=1}^kx_j\mB_j)\to S$ is the anticanonical model (over $S$) of $(\mY', \sum_{j=1}^kx_j\mB'_j)\to S$ for any $(x_1,...,x_k)\in F^\circ$, which implies that $(\mY, \sum_{j=1}^kx_j\mB_j)\to S$ is an $\bR$-Gorenstein family of log Fano pairs for any $(x_1,...,x_k)\in F^\circ$. The proof is complete.
\end{proof}

\section{Semi-algebraic K-semistable domains}

In this section, we aim to show that $\Kss(X, \sum_{j=1}^kD_j)_P$ is semi-algebraic for any $(X, \sum_{j=1}^k D_j)\in \widehat{\mG}$, under the assumption that $\widetilde{\mG}$ is log bounded. Suppose $P=\cup_s P_s$ is the finite rational polytope chamber decomposition in Proposition \ref{prop: uni fano}, we have the following characterization:
$$\Kss(X, \sum_{j=1}^kD_j)_P=\bigcup_F\  \Kss(X, \sum_{j=1}^kD_j)_{F^\circ},$$
where $F$ runs through all faces of all chambers $P_s$'s. Thus it is enough to characterize $\Kss(X, \sum_{j=1}^kD_j)_{F^\circ}$. 

The following proposition is one of the main achievement due to \cite[Proposition 8.8]{LZ24b}, where the semi-algebraicity of the K-semistable domain is obtained under a volume condition. We heavily rely on this preparation to prove the full semi-algebraic statement (with the volume condition replaced by a log bounded condition) in Proposition \ref{prop: semi-alg-2}. This means we do not give a new proof of the semi-algebraic statement in \cite[Proposition 8.8]{LZ24b}, though the volume condition is a special case of the log bounded condition.

\begin{proposition}\label{prop: semi-alg-1}
Given a couple $(X, \sum_{j=1}^kD_j)\in \widehat{\mG}$ and a rational polytope $Q\subset P$. Suppose $(X, \sum_{j=1}^kx_jD_j)$ is log Fano for any $(x_1,...,x_k)\in Q$, then $\Kss(X, \sum_{j=1}^kD_j)_Q$ is a semi-algebraic set. Moreover, there exists a finite rational polytope chamber decomposition $Q=\cup_s Q_s$ satisfying the following property: for any face $q$ of any chamber $Q_s$, the restricted K-semistable domain $\Kss(X, \sum_{j=1}^kD_j)_{q^\circ}$ (if non-empty) is cut out by finitely many algebraic equations $G_{q^\circ,l}(x_1,...,x_k)\geq 0$ (indexed by $l$), where each $G_{q^\circ,l}$ is of the following form for some prime divisor $E_{q^\circ,l}$ over $X$:
$$G_{q^\circ,l}(x_1,...,x_k)=\vol(-K_X-\sum_{j=1}^kx_jD_j)\cdot\left(A_{X, \sum_{j=1}^kx_jD_j}(E_{q^\circ,l})-S_{X, \sum_{j=1}^kx_jD_j}(E_{q^\circ,l})\right). $$
In particular, $\Kss(X, \sum_{j=1}^kD_j)_Q$ can be characterized by the following set
$$\Kss(X, \sum_{j=1}^kD_j)_Q:=\bigcap_{q, l}\  \{(x_1,...,x_k)\in Q\ |\ \text{$G_{q^\circ, l}(x_1,...,x_k)\geq 0$}\},$$
where $q$ runs over all the faces of all the chambers in the decomposition $Q=\cup_s Q_s$, and $l$ runs over all the indexes corresponding to $q^\circ$. 
\end{proposition}

Note that in this proposition we do not assume that $\widehat{\mG}$ is log bounded.

\begin{proof}
First note that $\vol(-K_X-\sum_{j=1}^kx_jD_j)$ and $\delta(X, \sum_{j=1}^kx_jD_j)$ are both continuous on $(x_1,...,x_k)\in Q$ (e.g. \cite[Proposition 2.10]{LZ24a}). Thus there exist positive real numbers $\epsilon_0$ and $\delta_0$ (which may be assumed $<1$) such that 
$$\vol(-K_X-\sum_{j=1}^kx_jD_j)\geq \epsilon_0\ \ \text{and}\ \ \  \delta(X, \sum_{j=1}^kx_jD_j)\geq \delta_0$$
for any $(x_1,...,x_k)\in Q$. The rest of the proof is almost identical to that of \cite[Proposition 8.8]{LZ24b}. We outline it below.

\

\textit{Step 1.} We construct a set of couples $\mA(\delta_0, \epsilon_0)$ consisting of $(Y, \sum_{j=1}^kB_j)$ satisfying: 
\begin{enumerate}
\item $Y$ is normal projective of dimension $d$ and $B_j$'s are effective Weil divisors;
\item $(Y, \sum_{j=1}^kc_jB_j)$ is a weak log Fano pair for some vector $(c_1,...,c_k)\in Q$ with $\delta(Y, \sum_{j=1}^kc_jB_j)\geq \delta_0$;
\item $\vol(-K_Y-\sum_{j=1}^kx_jB_j)\geq \epsilon_0$ for any $(x_1,...,x_k)\in Q$.
\end{enumerate}
We know that $\mA(\delta_0, \epsilon_0)$ is log bounded (e.g. \cite[Section 7]{LZ24b}), and there exists a finite rational polytope chamber decomposition $Q=\cup_s Q_s$ such that 
\begin{enumerate}
\item for any couple $(Y, \sum_{j=1}^kB_j)\in \mA(\delta_0, \epsilon_0)$ and any face $q$ of any chamber $Q_s$, if $(Y, \sum_{j=1}^kc_jB_j)$ is log Fano (resp. weak log Fano) for some $(c_1,...,c_k)\in q^\circ$, then $(Y, \sum_{j=1}^kx_jB_j)$ is log Fano (resp. weak log Fano) for any $(x_1,...,x_k)\in q^\circ$ (\cite[Proposition 7.2]{LZ24b});
\item for any family of couples $(\mY, \sum_{j=1}^k\mB_j)\to Z$ over a smooth base $Z$ such that every fiber is contained in $\mA(\delta_0, \epsilon_0)$, if for some $(c_1,...,c_k)\in q^\circ$, $(\mY, \sum_{j=1}^kc_j\mB_j)\to Z$ is an $\bR$-Gorenstein family of log Fano pairs with $\delta(\mY_t, \sum_{j=1}^kc_j\mB_{j,t})\geq \delta_0$ for every $t\in Z$, then $(\mY, \sum_{j=1}^kx_j\mB_j)\to Z$ is an $\bR$-Gorenstein family of log Fano pairs for any $(x_1,...,x_k)\in q^\circ$ (\cite[Proposition 7.3]{LZ24b}).
\end{enumerate}

\

\textit{Step 2.} Fix a face $q$ of some chamber $Q_s$ and we only consider $(x_1,...,x_k)\in q^\circ$. We use the Hilbert scheme argument (see \cite[Step 2 of Proposition 8.8]{LZ24b} or the proof of Lemma \ref{lem: hilbert-para}) to construct finitely many families $\pi_i: (\mX^{(i)}, \sum_{j=1}^k\mD^{(i)}_j)\to Z_i$ such that $Z_i$ is smooth admitting an algebraic group action $G_i$, and 
$$(\mX^{(i)}, \sum_{j=1}^kx_j\mD^{(i)}_j)\to Z_i$$ 
is an $\bR$-Gorenstein family of log Fano pairs for any $(x_1,...,x_k)\in q^\circ$ and any index $i$; moreover, for any special test configuration $(\mX, \sum_{j=1}^kx_j\mD_j)\to \bA^1$ of $(X, \sum_{j=1}^kx_jD_j)$ with $\delta(\mX_0, \sum_{j=1}^kx_j\mD_{j,0})\geq \delta_0$, there exists an index $i$ and a morphism $\bA^1\to Z_i$ such that the special test configuration is obtained by the pull-back of $(\mX^{(i)}, \sum_{j=1}^kx_j\mD^{(i)}_j)\to Z_i$. 

\

\textit{Step 3.} We still focus on $(x_1,...,x_k)\in q^\circ$. For any $\bR$-Gorenstein family of log Fano pairs $\pi_i: (\mX^{(i)}, \sum_{j=1}^kx_j\mD^{(i)}_j)\to Z_i$ as in step 2, there is a CM line bundle $\Lambda(x_1,...,x_k)_i$ on $Z_i$. For any special test configuration of $(X, \sum_{j=1}^kx_jD_j)$ as in step 2, there is a one parameter subgroup $\rho: \bG_m\to G_i$ such that the generalized Futaki invariant $\Fut(\mX, \sum_{j=1}^kx_j\mD_j)$ can be identified with a weight function $\Fut_i(x_1,...,x_k; \rho)$. By \cite[Step 5 of Prop 8.8]{LZ24b}, for each index $i$, it suffices to check $\Fut_i(x_1,...,x_k; \rho)\geq 0$ for finitely many one parameter subgroups $\rho$'s of $G_i$, where each $\rho$ corresponds to a special test configuration of $(X, \sum_{j=1}^kx_jD_j)$.

Above all, for any face $q$ of any chamber $Q_s$ in step 1, the restricted K-semistable domain $\Kss(X, \sum_{j=1}^kD_j)_{q^\circ}$
is cut out by finite inequalities
$$\Fut_i(x_1,...,x_k; \rho_{il})\geq 0 $$
indexed by $i$ and $l$, where each $\rho_{il}$ corresponds to a special test configuration of the log Fano pair $(X, \sum_{j=1}^kx_jD_j)$ and 
$\Fut_i(x_1,...,x_k; \rho_{il})$ is the generalized Futaki invariant.

By now, we have shown that the restricted K-semistable domain $\Kss(X, \sum_{j=1}^kD_j)_{q^\circ}$ (if non-empty) is cut out by finitely many algebraic equations $G_{q^\circ,l}(x_1,...,x_k)\geq 0$ (indexed by $l$), where each $G_{q^\circ,l}$ is induced by some prime divisor $E_{q^\circ, l}$. We turn to the last statement. First, we have
$$\Kss(X, \sum_{j=1}^kD_j)_Q=\bigcup_{q}\  \Kss(X, \sum_{j=1}^kD_j)_{q^\circ}$$
and 
$$\Kss(X, \sum_{j=1}^kD_j)_{q^\circ}=\bigcap_{l}\  \{(x_1,...,x_k)\in q^\circ\ |\ \text{$G_{q^\circ, l}(x_1,...,x_k)\geq 0$}\}, $$
where $q$ runs over all the faces of all the chambers in the decomposition $Q=\cup_s Q_s$, and $l$ runs over all the indexes corresponding to $q^\circ$. But also note the following
$$\Kss(X, \sum_{j=1}^kD_j)_{q^\circ}=\bigcap_{p, l}\  \{(x_1,...,x_k)\in q^\circ\ |\ \text{$G_{p^\circ, l}(x_1,...,x_k)\geq 0$}\}, $$
where $p$ runs over all the faces of all the chambers in the decomposition $Q=\cup_s Q_s$, and $l$ runs over all the indexes corresponding to $p^\circ$. Therefore, 
$$\Kss(X, \sum_{j=1}^kD_j)_Q:=\bigcap_{q, l}\  \{(x_1,...,x_k)\in Q\ |\ \text{$G_{q^\circ, l}(x_1,...,x_k)\geq 0$}\},$$
where $q$ runs over all the faces of all the chambers in the decomposition $Q=\cup_s Q_s$, and $l$ runs over all the indexes corresponding to $q^\circ$. The proof is complete. 
\end{proof}

To test the K-semistability of a log Fano pair $(X, \sum_{j=1}^kc_jD_j)\in\widehat{\mG}$, one needs to test the non-negativity of the generalized Futaki invariants of all special degenerations (e.g. \cite{LX14, LZ24a}). To get the semi-algebraicity of the K-semistable domain, we must restrict the special degenerations to check. For instance, if we only test those special degenerations that are contained in $\widetilde{\mG}$,  which are already packed into finitely many well organized universal families $\phi_i: (\bm Y^{(i)}, \sum_{j=1}^k \bm B^{(i)}_j)\to W_{i}$ (see Section \ref{sec: 5} and Lemma \ref{lem: hilbert-para}), then these degenerations (together with the associated one parameter subgroups) can be controlled in a unified way, leading to the semi-algebraicity. 
Inspired by this, we first introduce a key concept of \textit{quasi-K-semistable domain} as an intermediate tool. 

\begin{definition}\label{def: quasi-kss}
Let $P=\cup_s P_s$ be the finite rational polytope chamber decomposition such that all the conditions in Proposition \ref{prop: uni fano} and Proposition \ref{prop: uni gorenstein} are satisfied. Let $F$ be a face of some chamber $P_s$ and $\phi_i: (\bm Y^{(i)}, \sum_{j=1}^k{\bm B}_j^{(i)})\to W_i$ the finite families of couples (indexed by $i$) associated to $F$, which are constructed in Lemma \ref{lem: hilbert-para}.

Given $(X, \sum_{j=1}^kD_j)\in \widehat{\mG}$ and $(x_1,...,x_k)\in F^\circ$ such that $(X, \sum_{j=1}^kx_jD_j)$ is a log Fano pair. We say $(X, \sum_{j=1}^kx_jD_j)$ is quasi-K-semistable if the generalized Futaki invariant of any special test configuration $(\mX, \sum_{j=1}^kx_j\mD_j)\to \bA^1$ of $(X, \sum_{j=1}^kx_jD_j)$, which is induced by the pull-back of some universal family $\phi_i: (\bm Y^{(i)}, \sum_{j=1}^k \bm B^{(i)}_j)\to W_{i}$, is non-negative. We denote the quasi-K-semistable domain as
$$\qKss(X, \sum_{j=1}^kD_j)_{F^\circ}:=\{(y_1,...,y_k)\in F^\circ\ |\ \text{$(X, \sum_{j=1}^ky_jD_j)$ is quasi-K-semistable}\}. $$
\end{definition}

It is direct by Definition \ref{def: quasi-kss} to have the following inclusion:
$$\Kss(X, \sum_{j=1}^kD_j)_{F^\circ} \subset \qKss(X, \sum_{j=1}^kD_j)_{F^\circ}.$$

We first prove the semi-algebraicity of the quasi-K-semistable domain under the log boundedness of K-semistable degenerations.

\begin{proposition}\label{prop: enough deg}
Notation as in Definition \ref{def: quasi-kss}. Assume $\widetilde{\mG}$ is log bounded. For any given $(X, \sum_{j=1}^kD_j)\in \widehat{\mG}$, the quasi-K-semistable domain $\qKss(X, \sum_{j=1}^kD_j)_{F^\circ}$ is semi-algebraic.
\end{proposition}

\begin{proof}

We only consider $(x_1,...,x_k)\in F^\circ$. We may assume $(X, \sum_{j=1}^kx_jD_j)$ is log Fano for any $(x_1,...,x_k)\in F^\circ$ and the  quasi-K-semistable domain $\qKss(X, \sum_{j=1}^kD_j)_{F^\circ}$ is non-empty. By Definition \ref{def: quasi-kss}, it suffices to test all special test configurations $(\mX, \sum_{j=1}^kx_j\mD_j)\to \bA^1$ of $(X, \sum_{j=1}^kx_jD_j)$, which is induced by the pull-back of some universal family $\phi_i: (\bm Y^{(i)}, \sum_{j=1}^k \bm B^{(i)}_j)\to W_{i}$.

By the construction in the proof of Lemma \ref{lem: hilbert-para}, $(\bm Y^{(i)}, \sum_{j=1}^kx_j\bm B^{(i)}_j)\to W_i$ is a universal family for each $i$ and the base $W_i$ admits a $\PGL$-group action, denoted by $G_i$.
Suppose $z\in W_i$ is the closed point such that 
$$(X, \sum_{j=1}^kD_j)\cong ({\bm Y}_z^{(i)}, \sum_{j=1}^k\bm B^{(i)}_{j,z}).$$ 
Denote $Z_i$ the closure of the orbit of $z\in W_i$ under the $G_i$-action. Up to a $G_i$-equivariant resolution, we may assume $Z_i$ is smooth. Let 
$$g_i: ({\bm Y}_{Z_i}^{(i)}, \sum_{j=1}^k\bm B^{(i)}_{j, Z_i})\to Z_i$$
be the pull-back of the family $\phi_i$ (see the first paragraph) along $Z_i\to W_i$. By the conditions (3) and (4) in Lemma \ref{lem: hilbert-para}, it follows that 
$$g_i: ({\bm Y}_{Z_i}^{(i)}, \sum_{j=1}^kx_j\bm B^{(i)}_{j, Z_i})\to Z_i $$
is an $\bR$-Gorenstein family of log Fano pairs for any $(x_1,...,x_k)\in F^\circ$. By the Hilbert scheme construction in the proof of Lemma \ref{lem: hilbert-para}, it follows that the special test configuration $(\mX, \sum_{j=1}^kx_j\mD_j)\to \bA^1$ obtained by the pull-back along $\bA^1\to W_i$ is induced by a one parameter subgroup $\rho: \bG_m\to G_i$ (the generalized Futaki invariant is compatible with the CM weight associated to the one parameter subgroup, e.g. \cite[Corollary 8.6]{LZ24b}). By Definition \ref{def: quasi-kss}, it suffices to test all such one parameter subgroups whose degenerations still correspond to fibers of $({\bm Y}_{Z_i}^{(i)}, \sum_{j=1}^kx_j\bm B^{(i)}_{j, Z_i})\to Z_i$ to test quasi-K-semistability of $(X, \sum_{j=1}^kx_jD_j)$. The rest of the proof is similar to step 3 of the proof of Proposition \ref{prop: semi-alg-1} (see also \cite[Steps 4-6 of Proposition 8.8]{LZ24b}). We just outline it here.

For the $\bR$-Gorenstein family of log Fano pairs $({\bm Y}_{Z_i}^{(i)}, \sum_{j=1}^kx_j\bm B^{(i)}_{j, Z_i})\to Z_i$, there is a CM line bundle $\Lambda(x_1,...,x_k)_i$ on $Z_i$. For the special test configuration $(\mX, \sum_{j=1}^kx_j\mD_j)\to \bA^1$ of $(X, \sum_{j=1}^kx_jD_j)$ induced by the one parameter subgroup $\rho: \bG_m\to G_i$ whose degeneration still corresponds to a fiber of $({\bm Y}_{Z_i}^{(i)}, \sum_{j=1}^kx_j\bm B^{(i)}_{j, Z_i})\to Z_i$, there is a weight function $\Fut_i(x_1,...,x_k; \rho)$ satisfying
$$\Fut(\mX, \sum_{j=1}^kx_j\mD_j)=\Fut_i(x_1,...,x_k; \rho).$$ 
By \cite[Step 6 of Proposition 8.8]{LZ24b}, it suffices to check $\Fut_i(x_1,...,x_k; \rho)\geq 0$ for finitely many one parameter subgroups $\rho$'s of $G_i$, where each $\rho$ corresponds to a special test configuration of $(X, \sum_{j=1}^kx_jD_j)$. This implies that $\qKss(X, \sum_{j=1}^kD_j)_{F^\circ}$ is semi-algebraic. The proof is complete.
\end{proof}

We are ready to prove the semi-algebraicity of the K-semistable domain under the log boundedness of K-semistable degenerations.

\begin{proposition}\label{prop: semi-alg-2}
Assume $\widetilde{\mG}$ is log bounded. Let $P=\cup_s P_s$ be the finite rational polytope chamber decomposition such that all the conditions in Proposition \ref{prop: uni fano} and Proposition \ref{prop: uni gorenstein} are satisfied.
For any given $(X, \sum_{j=1}^kD_j)\in \widehat{\mG}$, suppose $F$ is a face of some chamber $P_s$ such that the restricted K-semistable domain $\Kss(X, \sum_{j=1}^kD_j)_{F^\circ}$ is non-empty. Then $\Kss(X, \sum_{j=1}^kD_j)_{F^\circ}$ is the union of some connected components of $\qKss(X, \sum_{j=1}^kD_j)_{F^\circ}$. In particular, $\Kss(X, \sum_{j=1}^kD_j)_{F^\circ}$ is semi-algebraic. 
\end{proposition}

\begin{proof}
Since $\Kss(X, \sum_{j=1}^kD_j)_{F^\circ}$ is non-empty, we know that $(X, \sum_{j=1}^kx_jD_j)$ is log Fano for some $(x_1,...,x_k)\in F^\circ$, thus $(X, \sum_{j=1}^kx_jD_j)$ is log Fano for any $(x_1,...,x_k)\in F^\circ$. 
By the proof of Proposition \ref{prop: enough deg}, there exist finitely many polynomials $H_i(x_1,...,x_k)$ $(1\leq i\leq s)$ such that 
$$\qKss(X, \sum_{j=1}^kD_j)_{F^\circ}=\bigcap_{i=1}^s \{H_i(x_1,...,x_k)\geq 0\} \cap F^\circ, $$
where $H_i$ is the Fujita-Li invariant of $(X, \sum_{j=1}^kx_jD_j)$ associated to some prime divisor $F_i$ over $X$. We write $\qKss(X, \sum_{j=1}^kD_j)_{F^\circ}$ as a finite disjoint union of connected components as follows:
$$\qKss(X, \sum_{j=1}^kD_j)_{F^\circ} = \amalg_{\alpha\in I} R_\alpha,$$
where $I$ is a finite index set. Suppose $(c_1,...,c_k)\in R_{\alpha'}$ for some $\alpha'\in I$ satisfying $(c_1,...,c_k)\notin \Kss(X, \sum_{j=1}^kD_j)_{F^\circ}$. We aim to show $(x_1,...,x_k)\notin \Kss(X, \sum_{j=1}^kD_j)_{F^\circ}$ for any $(x_1,...,x_k)\in R_{\alpha'}$. Assume there exists $(a_1,...,a_k)\in R_{\alpha'}\cap \Kss(X, \sum_{j=1}^kD_j)_{F^\circ}$. Let $\gamma\subset R_{\alpha'}$ be a path connecting $(c_1,...,c_k)$ and $(a_1,...,a_k)$ (note that $R_{\alpha'}$ is connected and hence path connected by \cite[Proposition 8.11]{LZ24b}) with the two vectors being the two ending points.

Let $Q\subset F^\circ$ be a rational polytope containing the whole path $\gamma$. By Proposition \ref{prop: semi-alg-1}, $\Kss(X, \sum_{j=1}^kD_j)_Q$ is a semi-algebraic set, and it is cut out by finitely many algebraic equations $G_i(x_1,...,x_k)\geq 0$ ($1\leq i\leq l$), where each $G_i$ is the Fujita-Li invariant of $(X, \sum_{j=1}^kx_jD_j)$ with respect to some prime divisor $E_i$ over $X$. We may assume $\{G_i(x_1,...,x_k)=0\}$ intersects $\Kss(X, \sum_{j=1}^kD_j)_Q$ empty for $1\leq i\leq l'$, while $\{G_i(x_1,...,x_k)=0\}$ intersects $\Kss(X, \sum_{j=1}^kD_j)_Q$ non-empty for $l'+1\leq i\leq l$. Let us focus on $E_{l}$ and $G_l(x_1,...,x_k)$ (and it works similarly for other $l'+1\leq i\leq l$). Take $(b_1,...,b_k)\in \{G_l(x_1,...,x_k)=0\}\cap \Kss(X, \sum_{j=1}^kD_j)_{Q}$, then the log Fano pair $(X, \sum_{j=1}^kb_jD_j)$ is K-semistable, and $E_l$ induces a special test configuration $(\mX, \sum_{j=1}^kb_j\mD_j)\to \bA^1$ of $(X, \sum_{j=1}^kb_jD_j)$ such that the central fiber $(\mX_0, \sum_{j=1}^kb_j\mD_{j,0})$ is also a K-semistable log Fano pair, i.e., $(\mX_0, \sum_{j=1}^k\mD_{j,0})\in \widetilde{\mG}$. By Proposition \ref{prop: uni gorenstein}, $(\mX, \sum_{j=1}^kx_j\mD_j)\to \bA^1$ is a special test configuration for any $(x_1,...,x_k)\in F^\circ$; and by Lemma \ref{lem: university}, the special test configuration $(\mX, \sum_{j=1}^kx_j\mD_j)\to \bA^1$ is obtained via the pullback along $\bA^1\to W_i$ for some universal family $\phi_i: (\bm Y^{(i)}, \sum_{j=1}^k \bm B^{(i)}_j)\to W_{i}$.
By Definition \ref{def: quasi-kss}, we have the following $(\clubsuit)$
$$\qKss(X, \sum_{j=1}^kD_j)_{F^\circ}=\bigcap_{i=1}^s \{H_i\geq 0\} \bigcap_{i=l'+1}^l \{G_i\geq 0\}\cap F^\circ =\amalg_{\alpha\in I} R_\alpha,$$
and moreover we have the following $(\clubsuit\clubsuit)$
\begin{align*}
\Kss(X, \sum_{j=1}^kD_j)_{Q}&\ =\ \bigcap_{i=1}^s\{H_i\geq 0\}\bigcap_{i=l'+1}^{l}\{G_i\geq 0\}\bigcap_{i=1}^{l'}\{G_i\geq 0\} \cap Q\\
&\ =\ \bigcap_{i=1}^s\{H_i\geq 0\}\bigcap_{i=l'+1}^{l}\{G_i\geq 0\}\bigcap_{i=1}^{l'}\{G_i> 0\} \cap Q.
\end{align*}
Note that in one of the above formulation of $\Kss(X, \sum_{j=1}^kD_j)_{Q}$ we only consider $\{G_i>0\}$ for $1\leq i\leq l'$, because $\{G_i=0\}$ is assumed to intersect $\Kss(X, \sum_{j=1}^kD_j)_{Q}$ empty for $1\leq i\leq l'$.
Since $(a_1,...,a_k)\in \Kss(X, \sum_{j=1}^kD_j)_{Q}$, we have
$$G_1(a_1,...,a_k)>0, \ G_2(a_1,...,a_k)>0,\ ...\ ,\ G_{l'}(a_1,...,a_k)>0. $$
Since $(c_1,...,c_k)\notin \Kss(X, \sum_{j=1}^kD_j)_{Q}$, there exists at least one index $1\leq i\leq l'$, say $i=1$, such that $G_1(c_1,...,c_k)<0$. That is, we have
$$G_1(c_1,...,c_k)<0, \ G_2(c_1,...,c_k)\  ?\  0,\  ...\ ,\  G_{l'}(c_1,...,c_k) \ ? \ 0. $$
Here the notation $?$ means we cannot make sure which is bigger between the two sides, but it does not matter at all for our purpose. Now restricting $G_1,...,G_{l'}$ to the path $\gamma$, by the continuity, we know that there exists $(c_1',...,c_k')\in \gamma$ such that $G_i(c_1',...,c_k')\geq 0$ for every $1\leq i\leq l'$ and there exists an index $1\leq i_0\leq l'$ such that $G_{i_0}(c_1',...,c_k')=0$. By the characterizations $(\clubsuit)$ and $(\clubsuit\clubsuit)$, we derive $(c_1',...,c_k')\in \Kss(X, \sum_{j=1}^kD_j)_{Q}$. This leads to a contradiction since we assume that $\{G_{i}(x_1,...,x_k)=0\}$ intersects $\Kss(X, \sum_{j=1}^kD_j)_{Q}$ empty for any $1\leq i\leq l'$. 
The contradiction implies that if $(c_1,...,c_k)\in R_{\alpha'}$ satisfying $(c_1,...,c_k)\notin \Kss(X, \sum_{j=1}^kD_j)_{F^\circ}$, then $(x_1,...,x_k)\notin \Kss(X, \sum_{j=1}^kD_j)_{F^\circ}$ for any $(x_1,...,x_k)\in R_{\alpha'}$. Thus there exists a subset $I'\subset I$ such that 
$$\Kss(X, \sum_{j=1}^kD_j)_{F^\circ}=\amalg_{\alpha\in I'} R_\alpha.$$
In particular, $\Kss(X, \sum_{j=1}^kD_j)_{F^\circ}$ is semi-algebraic by \cite[Proposition 8.11]{LZ24b}. The proof is complete.
\end{proof}

By now, we get the two ingredients as promised in Introduction, i.e., (1) finiteness of the K-semistable domains (Proposition \ref{prop: finite kss domain}), and (2) semi-algebraicity of each K-semistable domain (Proposition \ref{prop: semi-alg-2}). With the two ingredients in hand, we next construct the desired semi-algebraic chamber decomposition.

\begin{theorem}\label{thm: kps cd}
Under the assumption that $\widetilde{\mG}$ and $\widetilde{\widetilde{\mG}}$ are log bounded, there exists a finite chamber decomposition $P=\cup_i A_i$ such that $A_i$'s are all semi-algebraic sets and mutually disjoint; moreover, for any $(X, \sum_{j=1}^kD_j)\in \widehat{\mG}$ and any piece $A_i$, the K-semistability and K-polystability of $(X, \sum_{j=1}^kx_jD_j)$ do not depend on $(x_1,...,x_k)\in A_i$.
\end{theorem}

\begin{proof}
We first address the K-semistable part. By Proposition \ref{prop: finite kss domain} and Proposition \ref{prop: semi-alg-2}, under the log boundedness of $\widetilde{\mG}$, the set of restricted K-semistable domains 
$$\bigg\{\Kss\bigg(X, \sum_{j=1}^kD_j\bigg)_P\ \bigg| \ \text{$(X, \sum_{j=1}^kD_j)\in \widehat{\mG}$}\bigg\}$$ 
is finite, and every $\Kss(X, \sum_{j=1}^kD_j)_P$ is semi-algebraic. Similarly, replacing $\mG$ with $\widetilde{\mG}$, under the log boundedness of $\widetilde{\widetilde{\mG}}$, the set of restricted K-semistable domains 
$$\bigg\{\Kss\bigg(X, \sum_{j=1}^kD_j\bigg)_P\ \bigg| \ \text{$(X, \sum_{j=1}^kD_j)\in \widehat{\widetilde{\mG}}$}\bigg\}$$ 
is finite, and every $\Kss(X, \sum_{j=1}^kD_j)_P$ is semi-algebraic. 
By refining these finite semi-algebraic K-semistable domains, we can construct a finite semi-algebraic chamber decomposition $P=\cup_i A_i$ such that, for any $(X, \sum_{j=1}^kD_j)\in \widehat{\widetilde{\mG}}$ (note that $\widehat{\mG}\subset \widehat{\widetilde{\mG}}$) and any piece $A_i$, the K-semistability of $(X, \sum_{j=1}^kx_jD_j)$ does not depend on $(x_1,...,x_k)\in A_i$.

Now we turn to the K-polystable part. Applying Proposition \ref{prop: uni gorenstein}, up to refining each piece $A_i$, we may assume the following property is satisfied: 
for any piece $A_i$, if $(\mY, \sum_{j=1}^kc_j\mB_j)\to S$ is an $\bR$-Gorenstein family of K-semistable log Fano pairs over a smooth base $S$ for some $(c_1,...,c_k)\in A_i$ with $(\mY_t, \sum_{j=1}^k\mB_{j,t})\in \widetilde{\mG}$ for all $t\in S$ and $(\mY_t, \sum_{j=1}^k\mB_{j,t})\in \widehat{\mG}$ for general $t\in S$, then $(\mY, \sum_{j=1}^kx_j\mB_j)\to S$ is an $\bR$-Gorenstein family of log Fano pairs for any $(x_1,...,x_k)\in A_i$. We claim for any $(X, \sum_{j=1}^kD_j)\in \widehat{\mG}$ that $(X, \sum_{j=1}^kx_jD_j)$ is K-polystable for all $(x_1,...,x_k)\in A_i$ if $(X, \sum_{j=1}^kc_jD_j)$ is K-polystable for some $(c_1,...,c_k)\in A_i$.

Suppose $(X, \sum_{j=1}^kc_jD_j)$ is K-polystable for some $(c_1,...,c_k)\in A_i$, then it follows that $(X, \sum_{j=1}^kx_jD_j)$ is K-semistable for any $(x_1,...,x_k)\in A_i$ by the first paragraph. Assume $(X, \sum_{j=1}^kb_jD_j)$ is not K-polystable for some $(b_1,...,b_k)\in A_i$. Then there exists a special test configuration $(\mX, \sum_{j=1}^kb_j\mD_j)\to \bA^1$ of $(X, \sum_{j=1}^kb_jD_j)$ such that the central fiber $(\mX_0, \sum_{j=1}^kb_j\mD_{j,0})$ is K-polystable (see \cite[Theorem 6.3]{LZ24a}). In particular, the test configuration is not of product type. Note that $(\mX_t, \sum_{j=1}^k\mD_{j,t})\in \widehat{\mG}$ for $t\ne 0$ and $(\mX_0, \sum_{j=1}^k\mD_{j,0})\in \widetilde{\mG}\subset \widehat{\widetilde{\mG}}$, we have $(\mX_0, \sum_{j=1}^kx_j\mD_{j,0})$ is K-semistable for any $(x_1,...,x_k)\in A_i$.  By the refinement in the second paragraph and applying Proposition \ref{prop: uni gorenstein}, $(\mX, \sum_{j=1}^kx_j\mD_j)\to \bA^1$ is a special test configuration of $(X, \sum_{j=1}^kx_jD_j)$ for any $(x_1,...,x_k)\in A_i$. 
In particular, $(\mX, \sum_{j=1}^kc_j\mD_j)\to \bA^1$ is a K-semistable degeneration of the K-polystable log Fano pair $(X, \sum_{j=1}^kc_jD_j)$, thus of product type. This is a contradiction and the proof is complete.
\end{proof}

\section{Wall crossing}

Theorem \ref{thm: kps cd} leads to a wall crossing picture for K-moduli spaces parametrizing objects in $\widehat{\mG}$. We give a brief description in this section. For any given vector $\bm x:=(x_1,...,x_k)\in P$, we consider the following subset
$$\widehat{\mG}^{\Kss}_{\bm x}:=\{(X, \sum_{j=1}^kD_j)\in \widehat{\mG}\ |\ \text{$(X, \sum_{j=1}^kx_jD_j)$ is a K-semistable log Fano pair}\}. $$
By \cite[Theorem 1.3]{LXZ22} complemented by \cite{LZ24a}, there exists a finite type Artin stack $\mM^{\Kss}_{\widehat{\mG}, \bm x}$ parametrizing the objects in $\widehat{\mG}^{\Kss}_{\bm x}$, and $\mM^{\Kss}_{\widehat{\mG}, \bm x}$ admits a proper good moduli space $M^{\Kps}_{\widehat{\mG}, \bm x}$. Let $\mM^{\Kss,\red}_{\widehat{\mG}, \bm x}$ be the moduli
functor associated to reduced schemes and $M^{\Kps,\red}_{\widehat{\mG}, \bm x}$ the corresponding reduced good moduli space. Denote $\phi_{\bm x}: \mM^{\Kss,\red}_{\widehat{\mG}, \bm x}\to M^{\Kps,\red}_{\widehat{\mG}, \bm x}$ the descending map. We have the following result.

\begin{theorem}\label{thm: wall crossing}
Under the assumption that $\widetilde{\mG}$ and $\widetilde{\widetilde{\mG}}$ are log bounded, there is a finite semi-algebraic chamber decomposition $P=\cup_i A_i$ such that 
for any given piece $A_i$, the descending map $\phi_{\bm x}: \mM^{\Kss,\red}_{\widehat{\mG}, \bm x}\to M^{\Kps,\red}_{\widehat{\mG}, \bm x}$ does not depend on $\bm x\in A_i$; moreover, for any $\bm x\in A_i$ and ${\bm x}_0\in \partial \bar{A_i}$ (i.e. ${\bm x}_0$ lies on the boundary of the closure of $A_i$), we have the following wall crossing diagram:
\begin{center}
	\begin{tikzcd}[column sep = 2em, row sep = 1.5em]
	 \mM^{\Kss,\red}_{\widehat{\mG}, \bm x} \arrow[d,"\phi_{\bm x}", swap] \arrow[rr,"\Psi_{\bm x, {\bm x}_0}"]&& \mM^{\Kss,\red}_{\widehat{\mG}, {\bm x}_0} \arrow[d,"\phi_{{\bm x}_0}",swap]\\
	 M^{\Kps,\red}_{\widehat{\mG}, \bm x}\arrow[rr,"\psi_{\bm x, {\bm x}_0}"]&& M^{\Kps,\red}_{\widehat{\mG}, {\bm x}_0},
	\end{tikzcd}
\end{center}
where 
$\psi_{\bm a, \bm a_0}$ is the induced proper morphism produced by the universal property of good moduli spaces. 
\end{theorem}

\begin{proof}
Implied by Theorem \ref{thm: kps cd}.
The wall crossing morphism $\Psi_{\bm x, {\bm x}_0}$ is described in \cite[Section 9.2]{LZ24b}.
\end{proof}

\noindent
\subsection*{Acknowledgement}

We would like to thank the referees for insightful suggestions and comments, which improve the exposition very much. The author is supported by an NSFC grant (No. 12501058) and a grant from Xiamen University (No. X2450214).

\bibliography{reference.bib}
\end{document}